\documentclass[12pt]{article}
\usepackage[utf8]{inputenc}
\usepackage{amsmath, amssymb, amsthm}
\usepackage{booktabs}
\usepackage{graphicx}
\usepackage{float}
\usepackage{hyperref}
\usepackage{geometry}
\usepackage{setspace}
\usepackage{enumitem}
\usepackage{array}
\usepackage{multirow}
\usepackage{xcolor}
\usepackage{algorithm}
\usepackage{algpseudocode}
\usepackage{caption}
\usepackage{subcaption}
\usepackage[numbers,sort&compress]{natbib}
\usepackage{titlesec}
\usepackage{afterpage}

\hypersetup{
    colorlinks=true,
    linkcolor=blue!60!black,
    citecolor=blue!60!black,
    urlcolor=blue!60!black
}

\theoremstyle{definition}
\newtheorem{definition}{Definition}[section]
\newtheorem{proposition}{Proposition}[section]
\newtheorem{remark}{Remark}[section]

\titleformat{\section}{\large\bfseries}{\thesection}{1em}{}
\titleformat{\subsection}{\normalsize\bfseries}{\thesubsection}{1em}{}
\titleformat{\subsubsection}{\normalsize\itshape}{\thesubsubsection}{1em}{}

\title{\textbf{Recursive Governance: A Graph-Theoretic Framework for \\ Risk Propagation and Drift Detection in Agentic AI Systems}}
\author{Advaith Nila Narayanan, The University of Massachusetts, Amherst	\footnote{anilanarayan@umass.edu} \\ Sriram Nagaraj, PhD \footnote{dr.sri.nagaraj@gmail.com, Corresponding author} }
\date{}

\begin{document}

\maketitle
\thispagestyle{empty}

\begin{abstract}
\noindent
As financial institutions transition from traditional predictive models to
autonomous agentic systems, the static model inventory requirements of traditional model risk management (MRM) face structural obsolescence. These
systems introduce a class of risk---\emph{Agency Risk}---that existing MRM frameworks are ill-equipped to quantify:
non-deterministic reasoning paths, dynamic tool invocation, and cascading
inter-agent dependencies that cannot be fully enumerated at registration time.
Prior treatments of LLM governance address these symptoms in isolation---output
bias testing here, hallucination benchmarking there---without a unifying
structural account of \emph{why} agentic systems break the inventory paradigm
or a mechanism for containing the resulting risk once it manifests. The
central novelty of this paper is to close that gap with a single, internally
consistent governance loop rather than a collection of point fixes: we are,
to our knowledge, among the first to (i)~formalize the agent inventory itself as a
directed graph over which validation failures provably propagate, (ii)~give a
materiality score that is provably immune to tool-count inflation, and
(iii)~ground reasoning-trajectory drift detection in a calibration procedure
that matches the production test statistic rather than approximating it.

This paper proposes a dynamic \emph{Inventory-as-Code} (IaC) governance loop
that treats the model inventory as a living architectural component rather
than a periodic documentation artifact. We make four principal contributions.
First, we introduce a calibrated \emph{Degree of Autonomy} (DoA) materiality
score with an explicit, taxonomized tool-complexity weighting scheme that
addresses the dominance problem of naive additive risk formulations, and we
prove that tool breadth alone can never lift an agent's tier---closing a
gaming vector that additive scores leave open. Second,
we construct the agent inventory as a Directed Acyclic Graph (DAG) and define
a formal \emph{Composite Risk Propagation} algorithm under which upstream
validation failures induce risk penalties on all reachable descendants; we
further show that the blocking condition is properly decoupled from an
agent's inherent risk tier, so that high-tier agents are not
auto-quarantined by their own risk scores, and we prove that the resulting
blast radius equals exactly the transitive descendant set of a failure---a
containment guarantee that holds by construction rather than as a numerical
byproduct of penalty calibration. Third, we develop a
\emph{Trajectory Monitoring} protocol based on distributional cosine drift
across ensembled Chain-of-Thought (CoT) embeddings, with an explicit
procedure for constructing and certifying the \emph{Golden Path} baseline,
a matched-bootstrap calibration that we show is necessary to avoid a severe
false-positive artifact in the naive alternative, and a two-stage response
that separates legitimate reasoning variation from detrimental drift without
demanding that a single exceedance event trigger an irreversible action.
Fourth, we address practical complications largely absent from the prior
literature: LLM base-model version changes, latent feedback loops in
nominally acyclic agent graphs, and a two-pass execution structure that
stages validation ahead of risk propagation. Throughout, we motivate each
design choice against the specific failure mode it is meant to close, rather
than presenting the mechanisms as free-standing formalism.

We illustrate the framework with a fully reproducible simulation study on
a 17-agent inventory (random seed 42; code provided as a companion script).
A single Tier~1 failure cascades to block exactly the 6 agents in its
transitive descendant set while leaving the remaining 10 agents, including
the 3 other Tier~1 roots, operationally validated. Calibrating the trajectory
monitor with a matched bootstrap that mirrors the production test statistic,
we obtain a clean separation: all in-distribution validated agents remain
within bounds while every agent subject to genuine trajectory drift is
hard-flagged. We further show that the naive half-split calibration common in
drift detection instead soft-flags roughly 90\% of in-distribution Tier~1
agents---a sample-size artifact that the matched calibration removes.
\end{abstract}

\newpage
\tableofcontents

% ─────────────────────────────────────────────────────────────────────────────
\section{Introduction}
% ─────────────────────────────────────────────────────────────────────────────

Traditional model risk management (MRM) treats a \emph{model} as a quantitative
method, system, or approach that applies statistical, economic, financial, or
mathematical theories, techniques, and assumptions to process input data into
quantitative estimates. Mature MRM practice requires an institution to maintain
a model inventory that captures all models in use and assigns each a risk tier
commensurate with its materiality.

For over a decade, this framework had governed traditional predictive
models---credit scorecards, prepayment models, VaR engines---with reasonable
adequacy. The model's inputs, transformations, and outputs are knowable at
registration time. Validation is a point-in-time exercise. Inventory updates
are periodic and manual.

Large Language Model (LLM)-based agentic systems violate many, if not all,  of these
assumptions. An agent does not produce a single output from a fixed
transformation; it reasons iteratively, invokes external tools, delegates
sub-tasks to other agents~\citep{schick2023, yao2023, autogen2023}, and may arrive at different action sequences given
identical prompts. Its ``model'' cannot be reduced to a weight matrix because
the effective computation includes retrieval corpora, tool APIs, orchestration
logic, and the stochastic sampling of the underlying LLM. Most critically, it
\emph{acts}: it writes to databases, executes trades, generates customer
communications, and calls downstream agents that themselves act.

\subsection*{Causes: Why Agentic Systems Are Structurally Different}

Four properties, largely absent from the traditional model taxonomy, jointly
drive this departure.

\begin{itemize}[leftmargin=2em]
    \item \textbf{Non-determinism at the reasoning level, not just the output
    level.} Traditional models are stochastic only in a bounded, well-characterized
    sense (e.g., sampling noise around a fitted estimate). An LLM agent's
    \emph{path} to an output---which tools it calls, in what order, with what
    intermediate justifications---can vary materially across otherwise identical
    invocations, because the sampling occurs at every reasoning step, not once at
    the end.
    \item \textbf{Dynamic, self-directed tool invocation.} An agent selects which
    external systems to query or act upon at runtime, based on its own
    interpretation of the task. The realized set of tool calls for a given
    session is therefore a function of the agent's reasoning, not a fixed
    input/output contract fixed at development time.
    \item \textbf{Delegation and inter-agent composition.} Multi-agent
    architectures allow one agent to invoke another as a subroutine
    \citep{autogen2023}. This means an agent's effective behavior depends on the
    current state and validation status of agents it did not directly configure
    and may not even be aware of, creating dependency chains that are invisible
    to any single-model view of risk.
    \item \textbf{Exposure to uncontrolled upstream change.} The underlying LLM
    backbone, its safety tuning, and any external retrieval corpora it draws on
    are typically controlled by a third-party provider and can change on a
    release cadence entirely outside the institution's own change-control
    process. A model that was fully validated yesterday may be running on
    materially different weights today without any internally-initiated model
    change event.
\end{itemize}

Each of these causes is why a static, registration-time inventory entry is the
wrong governance primitive for an agentic system: the very quantities a
traditional inventory records as fixed attributes---inputs, outputs, and
validation status---are, for an agent, functions of runtime state.

\subsection*{Effects: Why This Matters Operationally}

These structural causes translate into concrete operational and financial
consequences, which is why we treat Agency Risk as a distinct risk category
rather than an extension of existing model risk taxonomies.

\begin{itemize}[leftmargin=2em]
    \item \textbf{Cascading failure across dependency chains.} Because agents
    consume one another's outputs, a validation failure or behavioral
    degradation in one upstream agent silently propagates its error into every
    downstream consumer, potentially several hops removed from the point of
    failure and undetectable from any single agent's own test results.
    \item \textbf{Undetected reasoning drift beneath stable outputs.} An agent
    can preserve a stable output distribution---the quantity classical
    monitoring tests---while its underlying reasoning shifts to a different,
    potentially exploitable justification for that output (Section~\ref{sec:trajectory}). Output-only
    monitoring is therefore blind to an entire class of behavioral change.
    \item \textbf{Uncontrolled expansion of the action surface.} Because agents
    can write to production systems---executing trades, issuing credit
    decisions, sending customer communications---an ungoverned agent does not
    merely produce a wrong number for a human to catch; it can directly cause
    financial loss, data breach, or customer harm before any human review
    occurs.
    \item \textbf{Audit and attestation failure.} If the inventory cannot
    account for an agent's current dependency graph, tool access, and
    validation status in real time, the institution cannot truthfully attest
    that its models are governed, independent of whether any
    individual agent is in fact behaving safely. This is a compliance risk in
    its own right, separate from the underlying behavioral risk.
\end{itemize}

The existing MRM literature has begun to address LLM governance in
isolation~\citep{nist_airmf}---focusing on output bias testing, hallucination rates, and
distributional drift on model outputs. These efforts address individual
symptoms of the causes above but not their common structural source: they
treat each agent as if it were still a standalone model with a fixed
input/output contract. What is absent is a systemic treatment
of the \emph{inter-agent} risk structure: the graph of dependencies, the
propagation of validation failures, and the formal monitoring of reasoning
trajectories as first-class MRM objects. This paper fills that gap.

The paper proceeds as follows. Section~\ref{sec:background} reviews traditional MRM background and notes the specific provisions that require
extension in the generative AI (genAI) context. Section~\ref{sec:iac} introduces the Inventory-as-Code
architecture. Section~\ref{sec:doa} develops the Degree of Autonomy
materiality scoring framework. Section~\ref{sec:dag} formalizes the DAG
model of agent interdependence and derives the Composite Risk Propagation
algorithm. Section~\ref{sec:trajectory} develops the Trajectory Monitoring
protocol in full. Section~\ref{sec:practical} addresses practical
implementation challenges. Section~\ref{sec:simulation} presents the
simulation study. Section~\ref{sec:conclusion} identifies open problems and
concludes.

% ─────────────────────────────────────────────────────────────────────────────
\section{MRM Background and the Agentic Gap}
\label{sec:background}
% ─────────────────────────────────────────────────────────────────────────────

\subsection{MRM: Core Provisions}

Traditional MRM established three pillars of model risk management: (1)~\emph{model
development, implementation, and use}; (2)~\emph{model validation}; and
(3)~\emph{governance, policies, and controls}. These three pillars exist for a
specific reason: development/use controls limit the risk that a model is built
or deployed outside its intended scope; validation controls limit the risk
that a model's actual performance diverges from its assumed performance; and
governance controls limit the risk that either of the first two failures goes
undetected by the institution. Each pillar implicitly assumes that the object
being governed---the model---has a stable identity between governance
touchpoints: the same weights, the same input schema, the same output
contract, until an explicit change event occurs. The governance pillar is most
immediately implicated by agentic systems, precisely because agentic systems
violate this stable-identity assumption without necessarily triggering any of
the change events the governance pillar is designed to catch.

Standard MRM practice requires model inventories to capture: the model's
purpose and use, its inputs and outputs, the business line(s) it supports,
its materiality tier, its validation status, and its ongoing monitoring
status. For traditional models, each of these fields is statically knowable.
For agentic systems, several are dynamically determined at runtime:

\begin{itemize}[leftmargin=2em]
    \item \textbf{Purpose and use} may change without a model change event if
    the orchestration prompt is modified outside of change-control scope.
    \item \textbf{Inputs} include external API responses, retrieval-augmented
    document corpora, and outputs from upstream agents---none of which are
    fixed at registration.
    \item \textbf{Outputs} include not just numerical estimates but action
    sequences with real-world consequences.
    \item \textbf{Validation status} can become stale when the underlying LLM
    base model is updated by the provider---a change entirely outside the
    institution's change-control process.
\end{itemize}

\subsection{What MRM Does Not Address}

Standard MRM was not designed for systems that delegate. This is not an
oversight so much as a natural consequence of the era in which its provisions
were written: the pillars were designed against a population of models that
each had one owner, one input schema, and no ability to invoke one another.
Consequently the framework contains no
provisions for:

\begin{itemize}[leftmargin=2em]
    \item \textbf{Inter-model dependency tracking}: If Model A feeds Model B,
    traditional MRM would recommend awareness of the dependency but provides no formal
    protocol for propagating validation failures. This matters because the
    absence of a formal protocol leaves the propagation decision to ad hoc
    human judgment, which does not scale once dependency graphs exceed a
    handful of models and cannot execute at the speed agents themselves
    operate.
    \item \textbf{Non-deterministic outputs}: MRM guidance assumes that
    validation can establish a stable output distribution. However, agentic outputs
    are path-dependent. This is consequential because a single point-in-time
    validation pass certifies only the reasoning paths sampled during that
    pass; it says nothing about paths the agent has not yet taken but may take
    under different production inputs.
    \item \textbf{Tool access as a risk dimension}: A model that can only
    \emph{read} data is categorically different in risk profile from one that
    can \emph{write} to a trading system. MRM has no vocabulary for this,
    which means two agents that are otherwise identical in accuracy and scope
    can carry wildly different real-world consequences under a materiality
    scheme that only looks at outputs.
    \item \textbf{Emergent multi-agent behaviors}: No provision covers risk
    that arises from the \emph{interaction} of multiple
    individually-validated agents---the case where every agent passes its own
    behavioral test in isolation, yet the composed system produces an
    unintended outcome because each agent's test scope did not include the
    behavior of its counterparts.
\end{itemize}

\subsection{Why a Structural, Rather than Point, Fix Is Required}

A natural first response to each gap above is a point fix: add a dependency
field to the inventory schema, add a non-determinism disclosure to the model
documentation template, add a tool-access column. We argue this is
insufficient, for a reason that recurs throughout this paper: each gap is a
symptom of the same underlying cause---the inventory is treated as a record
of what a system \emph{is}, when for an agentic system the governance-relevant
facts are about what the system is currently \emph{doing} and \emph{connected
to}. Patching individual fields leaves the update mechanism unchanged: a
human-mediated, periodic review cycle that cannot keep pace with a system
whose dependency graph, tool exposure, and reasoning trajectory can all shift
between two review dates. Sections~\ref{sec:iac}--\ref{sec:trajectory}
therefore replace the record-keeping paradigm itself, rather than extending
its schema, which is why the remainder of this paper is organized around a
control loop rather than a checklist.

% ─────────────────────────────────────────────────────────────────────────────
\section{The Inventory-as-Code Architecture}
\label{sec:iac}
% ─────────────────────────────────────────────────────────────────────────────

\subsection{From Registry to Feedback Loop}

The conventional model inventory is a \emph{ledger}: a periodically-updated
table of models and their attributes. We argue that this paradigm is
architecturally incompatible with agentic systems for a fundamental reason:
the inventory can never be synchronized with the system state because the
relevant state variables change at agent runtime, not at human review cadence.

The \textbf{Inventory-as-Code} (IaC) paradigm treats the inventory as a
\emph{real-time feedback control system}. The inventory is not populated by
humans after the fact; it is an active participant in the agent's execution
lifecycle. Every state change in any agent---a failed behavioral test, a
drift exceedance, a dependency change---is immediately reflected in the
inventory graph and triggers downstream consequences without human
intermediation.

We term this arrangement \emph{recursive governance}: the inventory recursively
governs the very agents that constitute it, and any governance action on one
agent---a failure or a quarantine---propagates recursively along the dependency
graph to all of its transitive descendants (Section~\ref{sec:dag}), so the
control loop continually re-evaluates its own downstream state.

This is analogous to the Infrastructure-as-Code principle in DevOps, where
the system's desired state is expressed declaratively and enforcement is
continuous. Applied to model governance, it means the inventory is the
\emph{source of truth} for operational permissions: an agent that is not
\texttt{Validated} in the inventory cannot execute, regardless of its
technical deployment status.

\subsection{The Four-Phase IaC Cycle}

The IaC loop operates as a continuous cycle with four distinct phases. The
four-way split is deliberate: each phase closes a specific failure mode left
open by the others. Discovery alone would find agents but not assign them a
risk-commensurate governance regime; registration alone would assign a
regime but not verify the agent actually meets it; a validation gate alone
would verify point-in-time behavior but say nothing about whether that
behavior persists after deployment; and monitoring alone, without the
preceding three phases, would have no baseline to monitor against and no
agreed identity for the thing being monitored. The phases are therefore
sequential dependencies, not independent checklist items.

\textbf{Phase 1: Discovery.} Automated scanning of agent configuration
manifests---stored as version-controlled YAML or JSON
specifications---identifies the LLM backbone (model ID, provider, version
hash), the registered toolset, the system prompt hash, and declared upstream
dependencies. Discovery is triggered by any commit to the agent configuration
repository and runs on a scheduled cadence independent of commits.

\textbf{Phase 2: Registration.} Each discovered agent is assigned a Unique
Model Identifier (UMI) following the taxonomy in Section~\ref{sec:doa}, its
materiality tier is computed via the DoA formula, and its node is created in
the inventory DAG with edges to declared upstream dependencies. An agent
declaring a dependency on an unregistered upstream system is flagged for
manual review before receiving a \texttt{Validated} status.

\textbf{Phase 3: Validation Gate.} Before an agent may execute in any
environment above development, it must pass a tier-specific battery of
behavioral benchmarks administered by the \texttt{ValidationGuard}
subsystem. Tier~1 agents require adversarial prompt injection tests,
boundary condition tests on all tool-call parameters, and a minimum of 500
distinct CoT trajectory samples for Golden Path construction
(Section~\ref{sec:trajectory}). Tier~3 agents require only output
distribution checks. Passage is logged with a timestamp, the benchmark suite
version, and the system prompt hash.

\textbf{Phase 4: Ongoing Monitoring.} Post-deployment, the IaC system
continuously ingests runtime telemetry to update node statuses. Three
triggers can change a node's status without human intervention: (a) a Drift
Metric exceedance (Section~\ref{sec:trajectory}), (b) propagation of an
upstream failure (Section~\ref{sec:dag}), or (c) a base model version change
detected via provider API polling. Status changes propagate immediately
through the DAG.

% ─────────────────────────────────────────────────────────────────────────────
\section{Materiality Scoring: The Degree of Autonomy Framework}
\label{sec:doa}
% ─────────────────────────────────────────────────────────────────────────────

We now define the components that enter our empirical model that quantifies agent risk.

\subsection{Component Definitions}

\begin{definition}[Impact Score, $I$]
$I \in \{1,\ldots,10\}$ measures the maximum potential financial or
operational harm from an incorrect or adversarially-induced output.
Calibration benchmarks: $I=1$ (pure-read informational retrieval, no
downstream action); $I=5$ (generates recommendations consumed by a human
decision-maker with consequential authority); $I=10$ (direct write access
to a trading system or credit decisioning engine with no human-in-the-loop
confirmation). Values must be assigned by the business line owner and
countersigned by the Chief Risk Officer or delegate.
\end{definition}

\begin{definition}[Degree of Autonomy, $A$]
$A \in \{1,\ldots,5\}$ measures the degree to which the agent's action
sequence is governed by human oversight:

\medskip
\begin{center}
\begin{tabular}{cl}
\toprule
$A$ & Operational Description \\
\midrule
1 & Human-in-the-Loop: every action requires explicit approval \\
2 & Human-on-the-Loop: human may intervene; actions execute after timeout \\
3 & Supervised Autonomy: human reviews aggregated action logs daily \\
4 & Asynchronous Oversight: exceptions escalated; routine actions unreviewed \\
5 & Full Autonomy: no human review of individual actions \\
\bottomrule
\end{tabular}
\end{center}
\end{definition}

\begin{definition}[Weighted Tool Complexity Score, $\mathcal{T}$]
Each tool is assigned to a severity class $s \in \{1,2,3,4\}$ and the corresponding class weight is $w_s$. If $n_s$ is the number of \emph{registered} tools of severity class $s$ in the agent's declared toolset---not a runtime invocation count---then we set the 
aggregate score as:
\begin{equation}
\mathcal{T} = \min\!\left(\sum_{s} w_s\, n_s,\; T_{\max}\right)
\end{equation}
where $T_{\max} = 20$ caps
the contribution of tool count (cardinality). Note that $T_{\max} = 20$ is a design parameter chosen for our simulation study; a real-world implementation would calibrate it against operational data. 
\end{definition}

\subsection{Tool Severity Taxonomy}

\begin{table}[h]
\centering
\caption{Tool Severity Taxonomy for $\mathcal{T}$ Computation}
\label{tab:tool_taxonomy}
\begin{tabular}{@{}clcc@{}}
\toprule
Class & Tool Type & Weight $w_s$ & Examples \\
\midrule
1 & Read-only, internal data & 0.5 & Data warehouse queries, internal dashboards \\
2 & Read-only, external API  & 1.0 & Market data feeds, reference APIs \\
3 & Write-access, reversible & 2.5 & Draft generation, staging database writes \\
4 & Write-access, irreversible & 5.0 & Trade execution, credit decisions, wire transfers \\
\bottomrule
\end{tabular}
\end{table}

\subsection{Motivation and the Dominance Problem}

The simplest formulation of an agent risk score:\begin{equation}R = (I \times A) + \sum_{s} w_s\, n_s\end{equation}contains a structural flaw: the unbounded summation of tool
complexities can dominate the impact-autonomy product for any agent with a
large toolset, regardless of whether those tools are consequential. A
read-only indexer with dozens of endpoints (such as AG-17 in our inventory,
Section~\ref{sec:simulation}) should not be classified Tier~1 by virtue of
cardinality alone. We replace this with a weighted,
bounded formulation that separates the \emph{breadth} and \emph{severity} of
tool access.

\subsection{The Calibrated Inherent Risk Score}

\begin{equation}
\boxed{R = \alpha \cdot (I \times A) + (1 - \alpha) \cdot \mathcal{T}}
\label{eq:risk_score}
\end{equation}

where $\alpha \in (0,1)$ blends the impact-autonomy product against tool
access risk. We set $\alpha = 0.6$ throughout, so $R_{\max} = 0.6 \times 50
+ 0.4 \times 20 = 38.0$.

\begin{proposition}[Tier Assignment Rule]
With $\alpha=0.6$ and $T_{\max}=20$, the tier boundaries are:
\begin{equation}
\mathrm{Tier}(R) = \begin{cases}
1 & R \geq 0.75 \cdot R_{\max} = 28.5 \\
2 & 0.40 \cdot R_{\max} \leq R < 28.5 \quad (15.2 \leq R < 28.5)\\
3 & R < 0.40 \cdot R_{\max} = 15.2
\end{cases}
\end{equation}
\end{proposition}

\begin{remark}
An agent with $I=10$, $A=5$ satisfies $R \geq 0.6 \times 50 = 30 > 28.5$
and is automatically Tier~1 regardless of toolset. Conversely, the tool term
is bounded: its maximum contribution is $(1-\alpha)\,T_{\max} = 0.4 \times 20
= 8.0$, strictly below the Tier~2 cutoff of $15.2$, so \emph{tool breadth
alone can never lift an agent above Tier~3}. For instance AG-17, a read-only
agent with $A=1$, $I=2$ and 60 Class-1 tools, has raw tool score
$60 \times 0.5 = 30$ capped to $\mathcal{T} = \min(30,20) = 20$, giving
$R = 0.6 \times 2 + 0.4 \times 20 = 9.2$---correctly Tier~3 despite carrying
the largest toolset in the inventory (a naive additive score
$I{\times}A + \sum_{s} w_s n_s = 2 + 30 = 32$ would wrongly rank it near the
top tier).

Critically, tier assignment and \emph{runtime quarantine} are decoupled
in our framework. A Tier~1 agent with a high inherent risk score is not
automatically quarantined; it is quarantined only when upstream failures
propagate to it (Section~\ref{sec:dag}). Conflating the two would render
all high-tier agents inoperative by design.
\end{remark}

\subsection{Agent Inventory: Inherent Risk Scores}

Table~\ref{tab:risk_scores} reports the inherent risk scores for the
17-agent inventory used in our simulation study (Section~\ref{sec:simulation}).
Agents are grouped by their computed tier. All scores are reproducible from
the companion Python script with \texttt{SEED=42}.

\begin{table}[h]
\centering
\caption{Inherent Risk Scores for the 17-Agent Simulation Inventory ($\alpha=0.6$, $T_{\max}=20$)}
\label{tab:risk_scores}
\small
\begin{tabular}{@{}llccccl@{}}
\toprule
Agent ID & Tier & $I$ & $A$ & $\mathcal{T}$ & $R$ & Description \\
\midrule
AG-01 & 1 & 10 & 5 & 13.0 & 35.20 & Trade execution orchestrator \\
AG-02 & 1 & 10 & 5 &  9.5 & 33.80 & Credit decisioning engine \\
AG-03 & 1 & 10 & 4 & 12.5 & 29.00 & Capital reporter \\
AG-04 & 1 &  9 & 5 & 12.0 & 31.80 & Portfolio risk aggregator \\
\midrule
AG-05 & 2 &  7 & 4 &  8.0 & 20.00 & Position limit monitor \\
AG-06 & 2 &  6 & 4 &  9.0 & 18.00 & Market data normaliser \\
AG-07 & 2 &  8 & 3 &  9.5 & 18.20 & Counterparty exposure calc \\
AG-08 & 2 &  7 & 3 &  8.0 & 15.80 & Collateral valuation agent \\
AG-09 & 2 &  7 & 3 &  9.0 & 16.20 & Document classification \\
AG-10 & 2 &  7 & 3 &  9.5 & 16.40 & Limit-breach alerting \\
\midrule
AG-11 & 3 &  2 & 2 &  2.0 &  3.20 & Internal data retriever A \\
AG-12 & 3 &  2 & 1 &  1.5 &  1.80 & Internal data retriever B \\
AG-13 & 3 &  3 & 2 &  3.5 &  5.00 & Reference data lookup \\
AG-14 & 3 &  2 & 2 &  2.0 &  3.20 & Audit log reader \\
AG-15 & 3 &  3 & 2 &  3.0 &  4.80 & Reporting formatter \\
AG-16 & 3 &  2 & 1 &  1.0 &  1.60 & Status dashboard feeder \\
AG-17 & 3 &  2 & 1 & 20.0 &  9.20 & Enterprise log/document indexer\,$^{\dagger}$ \\
\bottomrule
\end{tabular}

{\footnotesize $^{\dagger}$AG-17 carries 60 Class-1 (read-only) tools; its raw
tool score $60\times0.5=30$ is capped to $T_{\max}=20$. It is the wide-but-shallow
control case that exercises the cap and blend (Section~\ref{sec:simulation}).}
\end{table}

% ─────────────────────────────────────────────────────────────────────────────
\section{Graph-Theoretic Modeling of Agent Interdependence}
\label{sec:dag}
% ─────────────────────────────────────────────────────────────────────────────

\subsection{The Agent Dependency Graph}

\begin{definition}[Agent Dependency Graph]
The agent inventory is modeled as a directed graph $G = (V, E)$ where:
\begin{itemize}[leftmargin=2em]
    \item $V = \{v_1, \ldots, v_n\}$ is the set of registered agents, each
    with associated metadata $(R_i, \mathrm{Tier}_i, \mathrm{Status}_i)$.
    \item $E \subseteq V \times V$, where $(v_i, v_j) \in E$ denotes that
    $v_j$ consumes the output of $v_i$ or invokes $v_i$ as a tool.
\end{itemize}
\end{definition}

$G$ is required to be acyclic (a DAG) for topological execution order to be
well-defined. Section~\ref{sec:cycles} addresses violations of this
assumption.

\subsection{Composite Risk Propagation}

The central insight of the DAG approach is that validation failures are not
local events. If agent $v_i$ fails its behavioral benchmark, every downstream
agent consuming $v_i$'s output is operating on potentially corrupted data.

\begin{definition}[Ancestor and Descendant Sets]
For agent $v_j \in V$:
\begin{align}
\mathcal{A}(v_j) &= \{v_i \in V : \exists \text{ directed path from } v_i \text{ to } v_j\} \\
\mathcal{D}(v_i) &= \{v_j \in V : \exists \text{ directed path from } v_i \text{ to } v_j\}
\end{align}
\end{definition}

\begin{definition}[Composite Risk Score]
\begin{equation}
R^C_j = R_j
  + \beta_F \cdot \bigl|\{v_i \in \mathcal{A}(v_j) : \mathrm{Status}_i = \mathtt{Failed}\}\bigr|
  + \beta_B \cdot \bigl|\{v_i \in \mathcal{A}(v_j) : \mathrm{Status}_i = \mathtt{Blocked}\}\bigr|
\label{eq:composite_risk}
\end{equation}
where $\beta_F > 0$ is the failed-ancestor penalty and $\beta_B > 0$ is the
blocked-ancestor penalty, with $\beta_F > \beta_B$. We set $\beta_F = 15.0$
and $\beta_B = 5.0$ throughout.
\end{definition}

\subsection{The Blocking Condition}
\label{subsec:blocking}

An important design choice is how to translate the propagation state into a
binary \texttt{Blocked} / \texttt{Validated} decision. Two properties are
required. First, the decision must be decoupled from an agent's \emph{own}
inherent risk $R_j$: a naive threshold on the absolute composite score,
$R^C_j \geq \tau$, would auto-quarantine high-tier agents even with no upstream
failure. Second, it should not hinge on the specific numeric values of the
penalties $\beta_F,\beta_B$, which are calibration choices for \emph{severity
reporting} and should not silently determine \emph{which} agents are
quarantined.

We therefore define blocking directly by \emph{contamination reachability}:
\begin{equation}
\mathrm{Status}_j = \mathtt{Blocked}
\iff \exists\, v_i \in \mathcal{A}(v_j) \ \text{with}\
\mathrm{Status}_i \in \{\mathtt{Failed}, \mathtt{Blocked}\},
\label{eq:blocking_condition}
\end{equation}
i.e.\ an agent (that does not itself fail) is quarantined exactly when its
ancestor set contains at least one \emph{contaminated} (Failed or Blocked)
agent. The composite score $R^C_j$ of Eq.~\eqref{eq:composite_risk} is retained
separately as a \emph{severity magnitude}---it quantifies how deeply an agent
is embedded in the cascade via its count of failed and blocked ancestors---but
it no longer gates the Blocked/Validated decision.

\begin{proposition}[Blast radius equals reachability from a failure]
\label{prop:blast}
Under Eq.~\eqref{eq:blocking_condition}, evaluated in any topological order, an
agent $v_j$ that passes its own behavioral test is \texttt{Blocked} if and only
if $\mathcal{A}(v_j)$ contains a \texttt{Failed} agent. Equivalently, the set of
Blocked agents is exactly the transitive descendant set $\mathcal{D}(v_f)$ of
the failed agent(s) $v_f$.
\end{proposition}

\begin{proof}
($\Leftarrow$) If some Failed $v_f \in \mathcal{A}(v_j)$ then $v_j$ has a
contaminated ancestor and, not having failed itself, is Blocked.
($\Rightarrow$) Suppose $v_j$ is Blocked. Then it has a contaminated ancestor
$v_a$; if $v_a$ is Failed we are done. Otherwise $v_a$ is Blocked, so by
Eq.~\eqref{eq:blocking_condition} $v_a$ has a contaminated ancestor strictly
earlier in the topological order. Iterating, this walk moves strictly backward
in a finite acyclic graph without revisiting a node; it cannot terminate at a
Blocked agent (each has a contaminated ancestor), so it terminates at a Failed
agent $v_f$. Since every step follows an ancestor edge, transitivity gives
$v_f \in \mathcal{A}(v_j)$.
\end{proof}

This is precisely the containment guarantee one wants---the blast radius of a
failure is exactly its transitive descendant set---but, unlike a threshold on
$\Delta R^C_j$, it holds \emph{by construction} rather than as a numerical
coincidence. A rule $\Delta R^C_j \geq \beta_F$ would additionally block any
agent that reaches $\beta_F$ from blocked ancestors alone (with $\beta_B=5$,
three blocked ancestors already sum to $\beta_F=15$), so its equivalence to
``one failed ancestor'' is not robust to the choice of penalties;
Eq.~\eqref{eq:blocking_condition} removes that dependence.

\begin{remark}
Eq.~\eqref{eq:blocking_condition} cleanly separates two distinct concepts:
\emph{tier} (a static property of the agent's design, measuring its potential
harm) and \emph{runtime status} (a dynamic property measuring its current
trustworthiness given the state of its upstream dependencies). Conflating them
by using an absolute threshold on $R^C_j$ would make Tier~1 agents ungovernable.
\end{remark}

\subsection{The Two-Pass Execution Structure}

We compute validation and propagation in two passes
(Algorithm~\ref{alg:twopass}): Pass~1 runs every agent's behavioral test and
records a tentative \texttt{Failed}/\texttt{Validated}$^*$ label, and Pass~2
then computes each agent's composite score and finalizes its status. We
emphasize that this structure is an \emph{implementation convenience, not a
correctness requirement}. Because an agent's composite score and blocking
decision depend only on its \emph{ancestors}
(Eqs.~\eqref{eq:composite_risk}--\eqref{eq:blocking_condition}), and a
topological order visits every ancestor before the agent itself, a single
interleaved pass already produces correct scores and statuses: when an agent is
visited, all of its ancestors---and hence their final
\texttt{Failed}/\texttt{Blocked} labels---are resolved. (Agents at the same
topological depth are mutually incomparable and never lie in one another's
ancestor sets, so one's failure never needs to be observed by the other.)
Separating validation from scoring is nonetheless useful in practice: the
behavioral tests are the expensive, potentially externally executed or
parallelized step, so staging them ahead of the cheap graph propagation keeps
the two concerns modular and makes propagation a pure function of the recorded
validation landscape.

\begin{algorithm}
\caption{Two-Pass Governed Workflow Execution}
\label{alg:twopass}
\begin{algorithmic}[1]
\Require DAG $G$, ValidationGuard $V_G$, severity penalties $\beta_F, \beta_B$
\State $\sigma \leftarrow$ \textsc{TopologicalSort}$(G)$
\State \Comment{Pass 1: Validation only}
\For{$v \in \sigma$}
    \If{$V_G$.\textsc{RunBehavioralTest}$(v)$ $=$ \textsc{False}}
        \State $\mathrm{Status}[v] \leftarrow \mathtt{Failed}$
    \Else
        \State $\mathrm{Status}[v] \leftarrow \mathtt{Validated}^*$ \Comment{tentative}
    \EndIf
\EndFor
\State \Comment{Pass 2: Composite risk scoring and final status}
\For{$v \in \sigma$}
    \State $R^C_v \leftarrow$ \textsc{CompositeRisk}$(v, G, \beta_F, \beta_B)$ \Comment{severity magnitude}
    \If{$\mathrm{Status}[v] \neq \mathtt{Failed}$ \textbf{and} $\mathcal{A}(v)$ has a \texttt{Failed}/\texttt{Blocked} agent}
        \State $\mathrm{Status}[v] \leftarrow \mathtt{Blocked}$
    \ElsIf{$\mathrm{Status}[v] = \mathtt{Validated}^*$}
        \State $\mathrm{Status}[v] \leftarrow \mathtt{Validated}$
    \EndIf
\EndFor
\Return $\{(\mathrm{Status}[v],\; R^C_v)\}_{v \in V}$
\end{algorithmic}
\end{algorithm}

Either way, because propagation reads only finalized ancestor labels, the
resulting statuses and composite scores are invariant to the particular (valid)
topological order chosen.

\subsection{Handling Quasi-Cyclic Architectures}
\label{sec:cycles}

Real agentic systems frequently violate the DAG assumption. Feedback agents,
self-correcting loops, and multi-turn orchestrators create directed cycles.
We handle these via \textbf{Loop Unrolling}: cycles are broken by
introducing time-indexed copies of the looping agents. Agent $v_i$ at
iteration $t$ is treated as a distinct node $v_i^{(t)}$, with edges
$(v_i^{(t)}, v_i^{(t+1)})$ representing the feedback path. A maximum
iteration depth $T_{\max}^{\mathrm{loop}}$ is enforced by the orchestration
layer and registered in the inventory; the unrolled graph remains a DAG
with $n \cdot T_{\max}^{\mathrm{loop}}$ nodes.

\subsection{Interdependence Reporting Snapshot}

At any timestamp, the IaC system generates an Interdependence Snapshot
(Table~\ref{tab:interdependence}) for MRM audit. The snapshot
is generated directly from the simulation study (Section~\ref{sec:simulation})
and shows a subset of the 17-agent inventory at the moment AG-01 fails its
behavioral test.

\begin{table}[h]
\centering
\caption{Agentic Model Interdependence Snapshot — Selected Agents
(Simulation seed=42; full inventory in Table~\ref{tab:simulation_results})}
\label{tab:interdependence}
\small
\begin{tabular}{@{}lllrrrl@{}}
\toprule
Agent ID & Tier & Upstream Deps & $R$ & $R^C$ & $\Delta R^C$ & Status \\
\midrule
AG-01 & 1 & None         & 35.20 & 35.20 & ---   & Failed    \\
AG-05 & 2 & AG-01        & 20.00 & 35.00 & +15.0 & Blocked   \\
AG-07 & 2 & AG-01        & 18.20 & 33.20 & +15.0 & Blocked   \\
AG-10 & 2 & AG-04, AG-05 & 16.40 & 36.40 & +20.0 & Blocked   \\
AG-15 & 3 & AG-09, AG-10 &  4.80 & 29.80 & +25.0 & Blocked   \\
AG-02 & 1 & None         & 33.80 & 33.80 & ---   & Validated \\
AG-06 & 2 & AG-02        & 18.00 & 18.00 & ---   & Validated \\
AG-12 & 3 & AG-06        &  1.80 &  1.80 & ---   & Validated \\
\bottomrule
\end{tabular}
\end{table}

The $\Delta R^C$ column surfaces the incremental penalty attributable to
upstream failures, enabling reviewers to immediately identify which agents
are at risk due to their graph position rather than their own validation
status.

% ─────────────────────────────────────────────────────────────────────────────
\section{Trajectory Monitoring: Formalizing the Golden Path}
\label{sec:trajectory}
% ─────────────────────────────────────────────────────────────────────────────

\subsection{The Inadequacy of Classical Drift Detection}

Standard model drift detection---Kolmogorov-Smirnov tests on output
distributions, Population Stability Index on input features---is
output-centric~\citep{gama2014, rabanser2019}. It asks: \emph{has the distribution of this model's outputs
changed?} For agentic systems, this question is necessary but insufficient.
Consider an agent whose output distribution is stable (e.g., it consistently
approves 73\% of credit applications) but whose \emph{reasoning path} has
shifted: it is now approving applications for different reasons in ways that
a new adversarial input distribution could manipulate. The output drift test
would not detect this. The Trajectory Monitoring protocol is designed to
detect it.

\subsection{Chain-of-Thought as an MRM Object}

We treat the agent's Chain-of-Thought (CoT)~\citep{wei2022}---the sequence of intermediate
reasoning steps generated before each action---as a first-class MRM 
artifact.

Recent work has shown that a model's stated CoT is not always a \emph{faithful}
account of the computation that produced its output: chain-of-thought
explanations can be systematically shaped by input features the model never
verbalizes, and measured faithfulness varies with model size and task
\citep{turpin2023, lanham2023}. This bears directly on what trajectory
monitoring can and cannot certify. Drift in the CoT embedding does \emph{not}
certify that an agent's hidden reasoning is correct or unchanged; it is a
monitorable \emph{proxy}---a necessary-but-not-sufficient signal whose value is
that a large, sustained shift in the distribution of \emph{stated} reasoning is
unlikely under stable, in-distribution operation and therefore warrants review.
We accordingly treat the CoT as an intermediate model artifact and a
pseudo-measure of the (un)explainability of agent behavior, complementary
to---not a replacement for---output-level drift testing and behavioral
validation.

For each invocation of agent $v_i$, the CoT is a sequence
$\mathbf{c} = (c_1, c_2, \ldots, c_k)$ of natural language steps. We
represent $\mathbf{c}$ as a vector in $\mathbb{R}^d$ via a frozen sentence
embedding model $\phi$:
\begin{equation}
\mathbf{e}(\mathbf{c}) = \phi\!\left(\bigoplus_{j=1}^{k} c_j\right) \in \mathbb{R}^d
\end{equation}
In the simulation we use $d = 128$, with embeddings L$_2$-normalized to the
unit sphere as is standard for sentence embedding models~\citep{reimers2019}.

\subsection{Constructing and Certifying the Golden Path}

\begin{definition}[Golden Path]
The Golden Path $\mathbf{g}_{v_i}$ for agent $v_i$ is the centroid of the
embedding distribution over a certified validation corpus of
$N_{\mathrm{val}}$ trajectories:
\begin{equation}
\mathbf{g}_{v_i} = \frac{1}{N_{\mathrm{val}}} \sum_{k=1}^{N_{\mathrm{val}}} \mathbf{e}(\mathbf{c}^{(k)})
\end{equation}
normalised to unit length.
\end{definition}

A Golden Path that is uncritically accepted from whatever validation runs
happen to produce risks certifying an unrepresentative or poorly-behaved
baseline, against which all future drift would be measured incorrectly. The
Golden Path must therefore satisfy four certification criteria before
registration, each targeting a distinct way the baseline could be
defective---incomplete task coverage, an internally inconsistent baseline,
too few samples to estimate the centroid reliably, and no human check that
the ``acceptable'' trajectories are in fact acceptable:

\begin{enumerate}[leftmargin=2em]
    \item \textbf{Coverage.} The prompt set must cover the agent's documented
    task distribution, including edge cases and adversarial inputs identified
    during development validation.
    \item \textbf{Intra-distributional consistency.}
    $\mathrm{tr}(\hat{\Sigma}_{\mathrm{val}}) \leq \sigma^2_{\max}$, where
    $\sigma^2_{\max}$ is a tier-specific variance ceiling.
    \item \textbf{Sample size.} $N_{\mathrm{val}} \geq 500$ (Tier~1),
    $\geq 200$ (Tier~2), $\geq 50$ (Tier~3).
    \item \textbf{Human review.} At least 10\% of validation CoTs must be
    manually annotated as ``acceptable reasoning'' by the validation team.
\end{enumerate}

\subsection{The Drift Metric and Permutation Calibration}
\label{subsec:driftmetric}

\begin{definition}[Trajectory Drift Metric]
Let $\mathbf{p}_{v_i}^{(t)}$ be the rolling centroid of the $N_{\mathrm{prod}}$
most recent production CoT embeddings. The Drift Metric at monitoring period
$t$ is:
\begin{equation}
D_t = 1 - \frac{\mathbf{g}_{v_i} \cdot \mathbf{p}_{v_i}^{(t)}}
                {\|\mathbf{g}_{v_i}\| \cdot \|\mathbf{p}_{v_i}^{(t)}\|}
\label{eq:drift}
\end{equation}
$D_t \in [0,2]$; $D_t = 0$ implies perfect alignment with the Golden Path.
\end{definition}

The drift threshold $\delta$ is calibrated by a \emph{matched bootstrap}~\citep{efron1993} that
reproduces the null distribution of the production statistic,
Eq.~\eqref{eq:drift}. That statistic compares an $N_{\mathrm{val}}$-sample
Golden Path centroid against an $N_{\mathrm{prod}}$-sample rolling window, so
under the null (in-distribution production) its variance is governed by
\emph{both} sample sizes, scaling as
$\tfrac{1}{N_{\mathrm{val}}} + \tfrac{1}{N_{\mathrm{prod}}}$. We therefore, per
replicate, resample a Golden Path at $N_{\mathrm{val}}$ \emph{and} a window at
$N_{\mathrm{prod}}$ from the certified validation corpus:

\begin{enumerate}[leftmargin=2em]
    \item Draw $N_{\mathrm{val}}$ indices with replacement from the validation
    corpus and form the unit-normalised centroid $\mathbf{g}^\star$.
    \item Independently draw $N_{\mathrm{prod}}$ indices with replacement and
    form the unit-normalised centroid $\mathbf{p}^\star$.
    \item Record $D(\mathbf{g}^\star, \mathbf{p}^\star)$.
    \item Repeat across $M = 1{,}000$ replicates to obtain the empirical null
    $\hat{F}_D$, and set $\delta = \hat{F}_D^{-1}(1 - \alpha_{\mathrm{fp}})$ for
    a false-positive target $\alpha_{\mathrm{fp}}$.
\end{enumerate}

Matching both sample sizes is essential. A naive null that splits
$N_{\mathrm{val}}$ into two equal halves compares two
$N_{\mathrm{val}}/2$-sample centroids, which does \emph{not} reproduce the
production statistic: when $N_{\mathrm{prod}} < N_{\mathrm{val}}$ the production
window is noisier than either half, so the half-split null understates the
production variance and yields a threshold that is far too tight. The effect is
severe at Tier~1, where $N_{\mathrm{val}} = 500 \gg N_{\mathrm{prod}} = 100$:
under the half-split null a perfectly in-distribution agent exceeds the nominal
$5\%$ threshold roughly $90\%$ of the time (Section~\ref{sec:simulation}). The
matched bootstrap removes this mismatch and restores false-positive control at
the intended $\alpha_{\mathrm{fp}}$; because it resamples with replacement it
also applies unchanged when $N_{\mathrm{val}} < N_{\mathrm{prod}}$ (as at
Tier~3).

We use $\alpha_{\mathrm{fp}} = 0.05$ (Tier~1), $0.01$ (Tier~2), and $0.10$
(Tier~3) as calibration defaults (Appendix~\ref{app:calibration}). Under the
matched calibration $\alpha_{\mathrm{fp}}$ is a pure false-positive
budget---the fraction of in-distribution monitoring windows expected to
soft-flag---set per tier to trade alert volume against detection sensitivity,
independent of $N_{\mathrm{val}}$.

\subsection{Distinguishing Legitimate Deviation from Unintended Drift}

A Drift Metric exceedance $D_t > \delta$ is necessary but not sufficient for
remedial action. $\delta$ is calibrated to a false-positive \emph{budget},
not a certainty threshold (Section~\ref{subsec:driftmetric}), so by
construction some exceedances are expected under entirely benign operation.
A single hard cutoff that immediately quarantines on any exceedance would
therefore misspend that budget in one of two directions: set $\delta$ tight
enough to catch genuine drift quickly and the agent inventory pays for it in
false quarantines of benign reasoning variation (e.g., a legitimately novel
but correct line of reasoning on an unusual input), each of which forces a
full Validation Gate to reverse; set $\delta$ loose enough to keep false
quarantines rare and detection of genuine drift is delayed. A single
threshold cannot simultaneously minimize both costs because it collapses two
different decisions---\emph{worth a human look} versus \emph{worth taking the
agent offline}---into one. We address this by decoupling the two decisions
into a two-stage response, so that the cheaper, reversible action (human
review) absorbs the bulk of the false-positive budget and the expensive,
disruptive action (quarantine) is reserved for exceedances that are either
large enough to be implausible under the null or persistent enough that
continued operation is no longer justified while awaiting review:

\textbf{Stage 1: Soft Flag.} When $D_t \in (\delta,\; 1.5\delta]$, the agent
status transitions to \texttt{Under Review}. This does \emph{not} trigger
downstream quarantine. A human validator reviews a random sample of recent
CoTs against the registered task scope.

\textbf{Stage 2: Hard Flag.} When $D_t > 1.5\delta$ or a Soft Flag persists
beyond $\tau_{\mathrm{review}}$ days, status transitions to
\texttt{Quarantined} and the downstream propagation algorithm runs. The agent
requires a full Validation Gate before reinstatement.

Under the matched calibration above, this two-stage rule cleanly separates the
two populations in the simulation study (Section~\ref{sec:simulation}): every
in-distribution validated agent remains \texttt{Within Bounds} while every agent
subject to a genuine drift injection is \texttt{Hard-Flagged}. Any residual Soft
Flags occur only at the per-tier $\alpha_{\mathrm{fp}}$ budget---confirming that
the threshold now controls the intended error rate, in contrast to the
$\sim$90\% false-positive rate produced by a sample-size-mismatched null.

\subsection{The LLM Version Change Problem}

When an LLM provider releases a new model version, the embedding geometry of
the agent's CoT trajectories may shift, rendering the Golden Path meaningless.
This is a direct consequence of the upstream-change cause identified in
Section~\ref{sec:background}: the provider's release is not an institutional
change event, so nothing in traditional change control would otherwise catch
it, yet it invalidates the very baseline the trajectory monitor depends on.
Treating it as a silent non-event would let the monitor keep comparing
production trajectories against a baseline computed under a superseded model,
producing either spurious drift alarms or, worse, a false sense of coverage.
We mandate instead: (1) the affected agent is immediately suspended and transitioned
to \texttt{Pending Revalidation}; (2) the Golden Path is invalidated and the
production trajectory buffer cleared; (3) a new Validation Gate run constructs
a new Golden Path and recomputes $\delta$; (4) the agent is reinstated only
upon successful Gate completion. No waiver is permitted for Tier~1 or
Tier~2 agents.

% ─────────────────────────────────────────────────────────────────────────────
\section{Practical Implementation Considerations}
\label{sec:practical}
% ─────────────────────────────────────────────────────────────────────────────

\subsection{Telemetry and Logging Architecture}

Every agent invocation must emit a structured telemetry event containing: the
UMI, invocation timestamp, input hash, CoT transcript, action taken, tools
invoked and their return values, and the output. This telemetry must be
written to an append-only, tamper-evident log accessible to the MRM function
independently of the business line, and retained for the institution's
specified model risk retention period (typically 7 years for typical
capital based models).

\subsection{Prompt Change Control}

Prompt engineering changes to an agent's system prompt constitute a model
change event and must trigger the Validation Gate. We recommend a
\textbf{Prompt Registry} in which system prompts are version-controlled,
each version is assigned a hash, and the currently-registered hash is
compared against the runtime hash at each invocation. A hash mismatch
generates an immediate alert and operationally suspends the agent pending
change-control review.

\subsection{Integration with Existing MRM Tooling}

The IaC system's CSV export maps directly to standard MRM inventory fields, with
the addition of the $R^C$ and $\Delta R^C$ columns. Status changes should
push automatically to the institution's model inventory database.

% ─────────────────────────────────────────────────────────────────────────────
\section{Simulation Study}
\label{sec:simulation}
% ─────────────────────────────────────────────────────────────────────────────

\subsection{Setup}

All results in this section are simulation based and require 
no external LLM inference. Embeddings are synthetic unit-sphere
vectors drawn from a Gaussian centered on a per-agent direction, simulating
the structure of real sentence embeddings.

The 17-agent inventory follows the structure of Table~\ref{tab:risk_scores}:
four Tier~1 root nodes (AG-01 through AG-04), six Tier~2 mid-layer agents
(AG-05 through AG-10), and seven Tier~3 agents (AG-11 through AG-17)---the last
of which, AG-17, is the wide-but-shallow control case (Section~\ref{sec:doa}).
The dependency graph contains 15 edges encoding the data-flow architecture
described in Section~\ref{sec:dag} (AG-17 is a standalone read-only node with no
dependencies). The graph is confirmed to be a DAG.

AG-01 is forced to fail its behavioral test by the \texttt{ValidationGuard},
simulating a Tier~1 root-node failure. All other agents pass. The two-pass
workflow (Algorithm~\ref{alg:twopass}) is then run with $\beta_F = 15.0$,
$\beta_B = 5.0$.

\subsection{Composite Risk Propagation Results}

Table~\ref{tab:simulation_results} reports the full two-pass output for all
17 agents, sorted by tier and then by agent ID within tier.

The \textbf{Circuit-Breaker} (CB) column reports the breaker state derived from
each agent's propagation result, in the spirit of the circuit-breaker pattern in
distributed systems~\citep{nygard2018}. The failed agent that trips the breaker is \texttt{Active};
a blocked agent with a directly \texttt{Failed} parent is \texttt{Triggered}
(first-hop containment); a blocked agent contaminated only through an
intermediate \texttt{Blocked} parent is \texttt{Quarantined} (deeper
containment); and a validated agent is on \texttt{Standby}. The
Triggered/Quarantined distinction records how far a blocked agent sits from the
originating failure, giving responders a natural triage order.

\begin{table}[h]
\centering
\caption{Two-Pass Workflow Results: 17-Agent Inventory under AG-01 Failure
(Seed=42; $\beta_F=15.0$, $\beta_B=5.0$)}
\label{tab:simulation_results}
\small
\begin{tabular}{@{}lllrrrlll@{}}
\toprule
Agent ID & Tier & Upstream Deps & $R$ & $R^C$ & $\Delta R^C$ & Status & CB & Traj.\ Flag \\
\midrule
AG-01 & 1 & None         & 35.20 & 35.20 & ---   & Failed    & Active      & Hard Flag \\
AG-02 & 1 & None         & 33.80 & 33.80 & ---   & Validated & Standby     & Within Bounds \\
AG-03 & 1 & None         & 29.00 & 29.00 & ---   & Validated & Standby     & Within Bounds \\
AG-04 & 1 & None         & 31.80 & 31.80 & ---   & Validated & Standby     & Within Bounds \\
\midrule
AG-05 & 2 & AG-01        & 20.00 & 35.00 & +15.0 & Blocked   & Triggered   & Hard Flag \\
AG-06 & 2 & AG-02        & 18.00 & 18.00 & ---   & Validated & Standby     & Within Bounds \\
AG-07 & 2 & AG-01        & 18.20 & 33.20 & +15.0 & Blocked   & Triggered   & Hard Flag \\
AG-08 & 2 & AG-02, AG-04 & 15.80 & 15.80 & ---   & Validated & Standby     & Within Bounds \\
AG-09 & 2 & AG-03        & 16.20 & 16.20 & ---   & Validated & Standby     & Within Bounds \\
AG-10 & 2 & AG-04, AG-05 & 16.40 & 36.40 & +20.0 & Blocked   & Quarantined & Hard Flag \\
\midrule
AG-11 & 3 & AG-05        &  3.20 & 23.20 & +20.0 & Blocked   & Quarantined & Hard Flag \\
AG-12 & 3 & AG-06        &  1.80 &  1.80 & ---   & Validated & Standby     & Within Bounds \\
AG-13 & 3 & AG-07        &  5.00 & 25.00 & +20.0 & Blocked   & Quarantined & Hard Flag \\
AG-14 & 3 & AG-08        &  3.20 &  3.20 & ---   & Validated & Standby     & Within Bounds \\
AG-15 & 3 & AG-09, AG-10 &  4.80 & 29.80 & +25.0 & Blocked   & Quarantined & Hard Flag \\
AG-16 & 3 & AG-09        &  1.60 &  1.60 & ---   & Validated & Standby     & Within Bounds \\
AG-17 & 3 & None         &  9.20 &  9.20 & ---   & Validated & Standby     & Within Bounds \\
\midrule
\multicolumn{9}{l}{\textbf{Summary:} 10 Validated, 6 Blocked, 1 Failed} \\
\bottomrule
\end{tabular}
\end{table}

\subsection{Key Findings}

\textbf{1. Proportional containment.} The single failure of AG-01 (1 of 4
Tier~1 roots) cascades to block exactly 6 agents: its two direct Tier~2
dependents (AG-05, AG-07), one second-hop Tier~2 agent (AG-10), two Tier~3
agents directly downstream of blocked Tier~2 nodes (AG-11, AG-13), and one
Tier~3 agent with mixed parentage (AG-15). The remaining 10 agents---including
the three other Tier~1 root nodes and the standalone AG-17---remain fully
operational. The blast radius is precisely the transitive descendant set
$\mathcal{D}(\text{AG-01})$, exactly as guaranteed by
Proposition~\ref{prop:blast}.

\textbf{2. Transitive blocking through validated intermediate nodes.} AG-10
is blocked despite having AG-04 (Validated) as one of its direct parents.
It is blocked because its ancestor set also contains AG-01 (via AG-05). This
demonstrates that the reachability rule correctly identifies the \emph{full}
contamination path, not merely the immediate parent status.
Similarly, AG-15 is blocked via the AG-09 $\to$ AG-10 path even though
AG-09 itself is validated.

\textbf{3. Penalty accumulation.} AG-10 accumulates $\Delta R^C = +20.0$
(one failed ancestor AG-01 contributing $\beta_F = 15$ and one blocked
ancestor AG-05 contributing $\beta_B = 5$), while AG-15 accumulates
$\Delta R^C = +25.0$ (AG-01 contributing 15, AG-05 and AG-10 each
contributing 5). This penalty accumulation gives MRM teams a quantitative
measure of how deeply embedded in the failure cascade a given agent is (Section~\ref{subsec:blocking}); the
composite score serves purely as this severity signal and no longer gates
quarantine.

\textbf{4. Tier independence of quarantine status.} All 6 blocked agents
span multiple tiers (Tier~2: AG-05, AG-07, AG-10; Tier~3: AG-11, AG-13,
AG-15), confirming that blocking is determined entirely by graph position
relative to the failure, not by the agent's inherent risk tier.

\begin{figure}[H]
\centering
\includegraphics[width=1.2\textwidth]{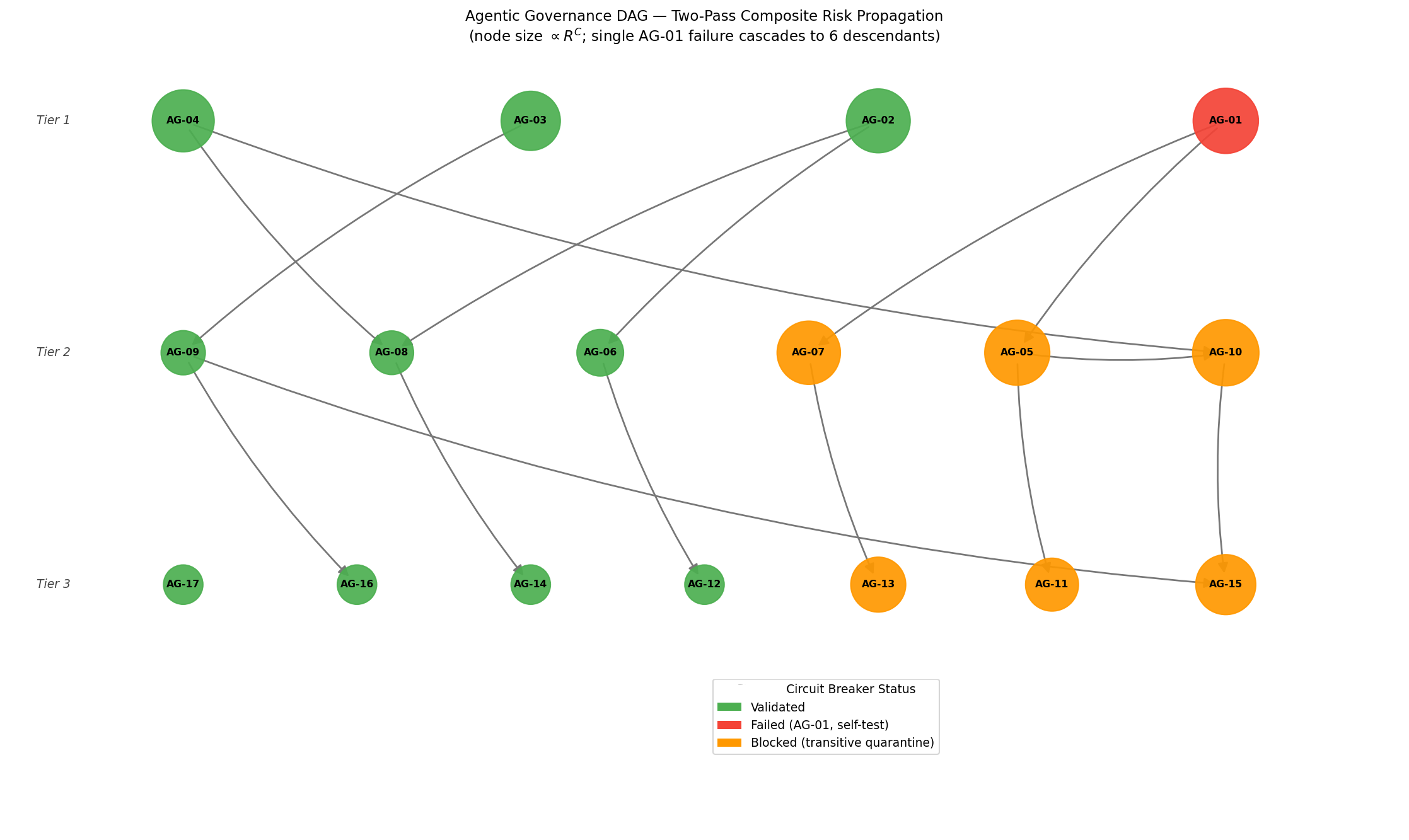}
\caption{Governance map simulation results: The governance DAG topology of agent interconnectedness and tiering. The figure shows the cascading of failures to descendent nodes and overall risk propagation.}
\label{fig:governance_map}
\end{figure}

\subsection{Trajectory Monitoring Results}

Drift metrics are computed with $N_{\mathrm{prod}} = 100$ production
invocations and $M = 1{,}000$ bootstrap replicates for the matched threshold
calibration of Section~\ref{sec:trajectory} (Appendix~\ref{app:calibration}).
Each agent's stochastic trajectory uses an independent, agent-keyed random
generator, so the results are invariant to processing order and to the presence
of other agents. With the matched calibration the monitor separates the two
populations cleanly.

\textbf{Validated agents.} All ten in-distribution validated agents are
classified \texttt{Within Bounds}, with drift ratios $D_t/\delta$ between
roughly $0.6$ and $0.9$. Crucially, the three surviving Tier~1 validated agents
(AG-02, AG-03, AG-04) now sit \emph{below} threshold at $D_t/\delta = 0.703$,
$0.829$, and $0.743$---in contrast to the half-split null, under which the same
agents soft-flag at ratios above $1.2$ (Section~\ref{sec:trajectory}). The
per-tier thresholds are $\delta \approx 0.020$--$0.022$ (Tier~1),
$0.028$--$0.031$ (Tier~2), and $0.045$--$0.054$ (Tier~3). AG-17, the
wide-but-shallow control agent, is likewise \texttt{Within Bounds}
($D_t/\delta = 0.883$).

\textbf{Failed and Blocked agents.} All seven agents subject to the drift
injection---the failed root AG-01 and its six blocked descendants (AG-05,
AG-07, AG-10, AG-11, AG-13, AG-15)---receive \texttt{Hard Flag} status, with
$D_t/\delta$ ratios from $3.3$ to $8.3$. In particular AG-13, which the
mis-calibrated half-split null had under-flagged as a mere Soft Flag, is now
correctly Hard-Flagged: the matched Tier~3 threshold ($\delta \approx 0.052$,
rather than the inflated $0.126$) no longer absorbs its drift.

Overall the monitor yields a clean confusion matrix on this inventory---$10$
Within Bounds, $0$ Soft, $7$ Hard---with all in-distribution agents passing and
all drifted agents caught. Residual soft false positives are expected only at
the per-tier $\alpha_{\mathrm{fp}}$ budget; this seed realizes none. The drift
injection magnitude of $0.60$ (orthogonal component) corresponds to a
substantial angular displacement on the unit sphere.

\begin{figure}[H]
\centering
\includegraphics[width=1.1\textwidth]{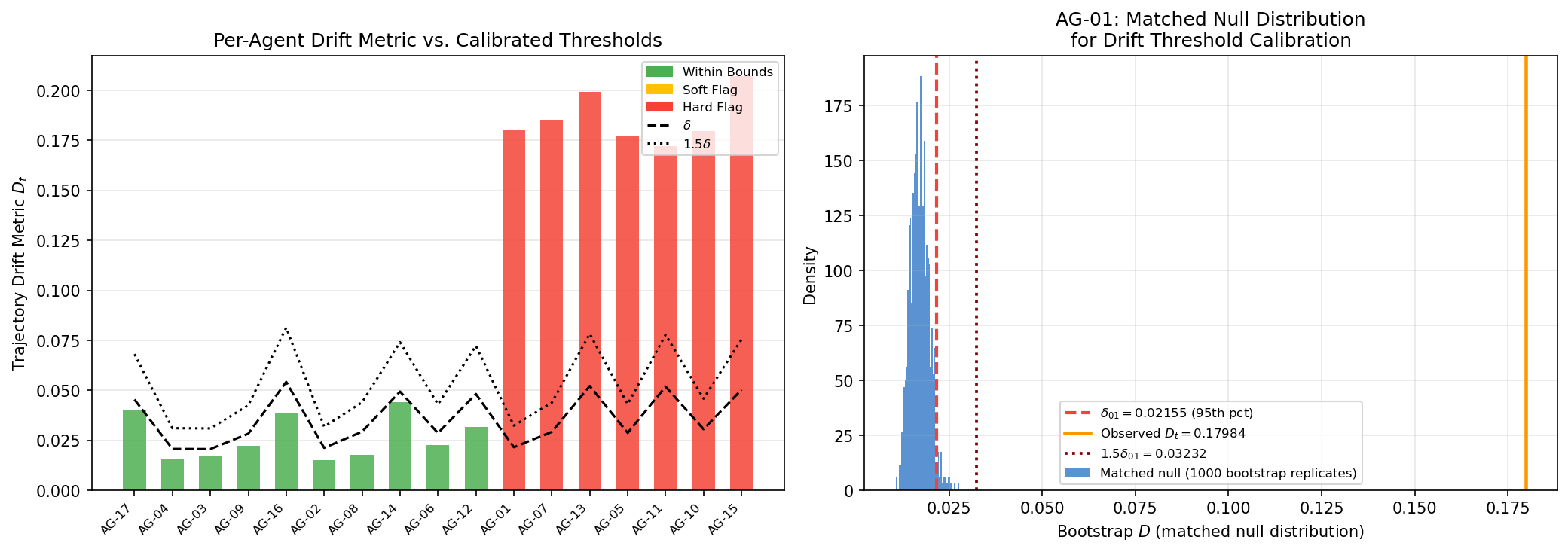}
\caption{Drift analysis simulation results: The figure shows per-agent trajectory drift, color coded by whether it is within bounds or soft/hard flagged (left panel). The right panel shows the case of agent AG-01 and its null distribution for drift threshold calibration.}
\label{fig:drift}
\end{figure}

% ─────────────────────────────────────────────────────────────────────────────
\section{Open Problems and Conclusion}
\label{sec:conclusion}
% ─────────────────────────────────────────────────────────────────────────────

\subsection{Open Problems}

\textbf{Embedding model stability.} The Trajectory Monitoring protocol
assumes the sentence embedding model $\phi$ is stable over time. If $\phi$
is updated, the Golden Path is invalidated by the same mechanism as a base
model update. Institutions should pin embedding model versions and treat
updates as a governance event requiring batch recomputation of all registered
Golden Paths.

\textbf{Multi-modal agentic outputs.} The CoT embedding approach is defined
for text-based reasoning chains. Agents reasoning over structured data, code,
or images require analogous representations in appropriate embedding spaces
with potentially different similarity metrics.

\textbf{Adversarial trajectory manipulation.} A sophisticated adversary who
knows the Golden Path could craft inputs that preserve $D_t < \delta$ while
manipulating agent behavior in semantically-aligned but functionally harmful
ways. Defense against this class of attack requires adversarial trajectory
testing during validation.

\textbf{Calibration under non-i.i.d.\ production.} The matched bootstrap of
Section~\ref{sec:trajectory} controls the false-positive rate under the
assumption that in-distribution production windows are exchangeable with the
validation corpus. Real production streams may instead exhibit temporal
autocorrelation, slow covariate drift, or heavier-tailed variation than the
validation set, for which an i.i.d.\ bootstrap null can be optimistic.
Extending the calibration to block-bootstrap or otherwise time-aware nulls,
and studying the sensitivity--burden trade-off of $\alpha_{\mathrm{fp}}$ across
real task distributions, remains valuable future work.

\subsection{Conclusion}

The governance of agentic AI systems using the principles of traditional MRM is not merely an
extension of existing MRM practice. It requires a
fundamental architectural shift: from the model inventory as a passive ledger
to an active, real-time control system that participates in the agent's
operational lifecycle.

This paper has developed that shift across four dimensions. The calibrated
DoA materiality score addresses the dominance problem in naive risk
aggregation and provides an explicit, auditable taxonomy for tool access risk.
The DAG-based composite risk propagation algorithm---with its two-pass
execution and blocking condition correctly decoupled from inherent risk
tier---formalizes the intuition that validation failures are systemic events.
The Golden Path trajectory monitoring framework provides a statistically
principled approach to drift detection that explicitly distinguishes
legitimate reasoning variation from genuine behavioral departure. And the
practical protocols for LLM version changes, feedback loop handling, and
prompt change control close the implementation gaps that would otherwise
render a theoretically sound framework operationally unworkable.

The simulation study demonstrates that the framework behaves as intended:
a single Tier~1 failure propagates to exactly its transitive descendant
set, leaving the rest of the inventory operational, while the trajectory
monitoring provides a continuous, quantitative signal of reasoning fidelity
that traditional output-distribution tests cannot supply.

\newpage
\appendix

% ─────────────────────────────────────────────────────────────────────────────
\section{Notation Reference}
\label{app:notation}
% ─────────────────────────────────────────────────────────────────────────────

\begin{table}[H]
\centering
\caption{Summary of Mathematical Notation}
\begin{tabular}{@{}ll@{}}
\toprule
Symbol & Definition \\
\midrule
$G = (V, E)$                & Agent dependency graph (DAG) \\
$v_i$                       & Individual agent node \\
$R_i$                       & Inherent Risk Score (Eq.~\ref{eq:risk_score}) \\
$R^C_i$                     & Composite Risk Score (Eq.~\ref{eq:composite_risk}) \\
$\Delta R^C_i$              & Composite risk increment: $R^C_i - R_i$ \\
$I$                         & Financial/Operational Impact Score ($\in \{1,\ldots,10\}$) \\
$A$                         & Degree of Autonomy ($\in \{1,\ldots,5\}$) \\
$\mathcal{T}$               & Weighted Tool Complexity Score \\
$w_s$                       & Weight for tool severity class $s$ \\
$\alpha$                    & Calibration parameter (= 0.6) \\
$\beta_F,\; \beta_B$        & Failed- and blocked-ancestor penalty coefficients \\
$\mathcal{A}(v_j)$         & Ancestor set of $v_j$ in $G$ \\
$\mathcal{D}(v_i)$         & Descendant set of $v_i$ in $G$ \\
$\mathbf{g}_{v_i}$          & Golden Path centroid for agent $v_i$ \\
$\mathbf{p}_{v_i}^{(t)}$   & Rolling production trajectory centroid at time $t$ \\
$D_t$                       & Trajectory Drift Metric at time $t$ (Eq.~\ref{eq:drift}) \\
$\delta$                    & Drift threshold (calibrated via matched bootstrap) \\
$\alpha_{\mathrm{fp}}$      & False-positive rate target for $\delta$ calibration \\
$\phi$                      & Sentence embedding function \\
$N_{\mathrm{val}}$          & Validation corpus size for Golden Path construction \\
\bottomrule
\end{tabular}
\end{table}

% ─────────────────────────────────────────────────────────────────────────────
\section{Recommended Calibration Parameters}
\label{app:calibration}
% ─────────────────────────────────────────────────────────────────────────────

\begin{table}[H]
\centering
\caption{Default Calibration Parameters by Tier}
\begin{tabular}{@{}lrrr@{}}
\toprule
Parameter & Tier 1 & Tier 2 & Tier 3 \\
\midrule
$N_{\mathrm{val}}$ (Golden Path samples)     & 500 & 200 & 50 \\
$\alpha_{\mathrm{fp}}$ (false-positive rate for $\delta$) & 0.05 & 0.01 & 0.10 \\
Hard Flag multiplier on $\delta$             & 1.5 & 1.5 & 1.5 \\
$\tau_{\mathrm{review}}$ (Soft Flag resolution, days) & 2 & 5 & 14 \\
Prompt hash check frequency                 & Per-invocation & Daily & Weekly \\
Base model version polling frequency        & Hourly & Daily & Weekly \\
$\beta_F$ (failed-ancestor penalty)         & \multicolumn{3}{c}{15.0 (uniform)} \\
$\beta_B$ (blocked-ancestor penalty)        & \multicolumn{3}{c}{5.0 (uniform)} \\
$\alpha$ (DoA risk blend weight)            & \multicolumn{3}{c}{0.6 (uniform)} \\
\bottomrule
\end{tabular}
\end{table}

% ─────────────────────────────────────────────────────────────────────────────

\end{document}